\documentclass[leqno]{siamltex704}
\usepackage{amsmath}
\usepackage{graphicx}
\usepackage{mathrsfs}
\usepackage{float}
\usepackage{amsfonts,amssymb}
\usepackage{dsfont}
\usepackage{pifont}
\usepackage{tikz}
\usepackage{wrapfig} % Allows in-line images
\usepackage{hyperref}
\usepackage{multirow}

\usepackage{mathtools}
\usepackage{subfigure}
\usepackage{color}
\def\argmin{\operatornamewithlimits{arg\ min}}
\def\sgn{\operatornamewithlimits{sgn}}
\def\diam{\operatornamewithlimits{diam}}

\def\T{{\mathcal T}}
\def\E{{\mathcal E}}

\def\argmin{\arg\min}
\def\d{{\rm div}}
\def\bn{{\bf n}}
\def\bq{{\bf q}}

\def\bv{{\bf v}}
\def\pT{{\partial T}}
\def\3bar{{|\!|\!|}}
\def\bbeta{{\boldsymbol\beta}}
\newtheorem{algorithm}{CFO Algorithm}[section]
\def\argmin{\operatornamewithlimits{arg\ min}}
\def\osc{\operatornamewithlimits{osc}}
\def\essensup{\operatornamewithlimits{essen\ sup}}
\title{A Conservative Flux Optimization Finite Element Method for Convection-Diffusion Equations}

\author{Yujie Liu\thanks{School of Data and Computer Science, Sun Yat-sen University, Guangzhou, 510275, China (liuyujie5@mail.sysu.edu.cn). The research of Liu was partially supported by Guangdong Provincial Natural Science Foundation (No. 2017A030310285), Shandong Provincial natural Science Foundation (No. ZR2016AB15) and Youthful Teacher Foster Plan Of Sun Yat-Sen University (No. 171gpy118),} \and Junping Wang
\thanks{Division of Mathematical Sciences, National Science Foundation, Arlington, VA 22230
(jwang@nsf.gov). The research of Wang was supported by the NSF IR/D program, while working at
National Science Foundation. However, any opinion, finding, and
conclusions or recommendations expressed in this material are those
of the author and do not necessarily reflect the views of the
National Science Foundation,}
\and Qingsong Zou
\thanks{School of Data and Computer
  Science, and Guangdong Province Key Laboratory of Computational Science, Sun Yat-sen
          University, Guangzhou 510006, China(mcszqs@mail.sysu.edu.cn).  The research of Zou was supported in part by the following Grants:
           the special project high performance computing of National Key Research and Development Program 2016YFB0200604, NSFC 11571384, Guangdong Provincial NSF 2014A030313179, the Fundamental Research Funds for the Central Universities 16lgjc80.}
}

\begin{document}

\maketitle

\begin{abstract}
This article presents a new finite element method for convection-diffusion equations by enhancing the continuous finite element space with a flux space for flux approximations that preserve the important mass conservation locally on each element. The numerical scheme is based on a constrained flux optimization approach where the constraint was given by local mass conservation equations and the flux error is minimized in a prescribed topology/metric. This new scheme provides numerical approximations for both the primal and the flux variables. It is shown that the numerical approximations for the primal and the flux variables are convergent with optimal order in some discrete Sobolev norms. Numerical experiments are conducted to confirm the convergence theory. Furthermore, the new scheme was employed in the computational simulation of a simplified two-phase flow problem in highly heterogeneous porous media. The numerical results illustrate an excellent performance of the method in scientific computing.
\end{abstract}



\pagestyle{myheadings}

%
%\noindent {\bf Mathematics Subject Classification (2010)} 65N15;
%65N30; 41A30

\section{Introduction}

This paper is concerned with the development of numerical methods for partial differential equations that maintain important conservation properties for the underlying physical variables. For simplicity, consider the model convection-diffusion equation that seeks an unknown function $u=u(x)$ satisfying
\begin{eqnarray}
-\nabla\cdot(\alpha\nabla u + \bbeta u)&=&f,\qquad {\rm in}\  \Omega  \label{ellipticbdy}\\
u&=&0,\qquad {\rm on}\ \partial\Omega \label{ellipticbc}
\end{eqnarray}
where $\Omega\subset {\bf R}^d (d=2,3)$ is a bounded polygonal ($d=2$) or polyhedral ($d=3$) domain with boundary $\partial\Omega$, and $\alpha=\{\alpha_{i,j}\}_{d\times d}$ is a symmetric, positive definite tensor; i.e., there exists a positive constant $\alpha_0$ such that
$$
\xi^T\alpha\xi\ge \alpha_0 \xi^T\xi, \qquad \forall \xi \in \Omega.
$$
In some applications, such as the flow of fluid in porous media governed by Darcy's law, the quantity of interest is often the flow velocity represented by $\bq = - (\alpha\nabla u + \bbeta u).$ With the velocity $\bq$, the equation \eqref{ellipticbdy} can be rewritten as
$\nabla\cdot\bq = f$, so that on any subdomain $T\subset\Omega$, from the divergence theorem one has
\begin{equation}\label{EQ:conservation}
\int_T \nabla\cdot\bq dx = \int_T f dx \ \ \Longleftrightarrow \int_{\partial T} \bq\cdot\bn ds = \int_T f dx,
\end{equation}
where $\bn$ is the outward normal direction of $\partial T$. The equations in \eqref{EQ:conservation}, especially the one on the right, characterize the mass conservation property for the porous media flow. The quantities that enter into the mass
conservation equation is the flux variable $q_n = \bq\cdot\bn$ on the boundary of any control element $T$.

A numerical scheme for the model convection-diffusion equation \eqref{ellipticbdy}-\eqref{ellipticbc} is said to be conservative if it provides a numerical solution, denoted by $u_h$, and an associated numerical flux $q_{n,h}$ on the boundary of a set of prescribed control elements $\T_h=\{T\}$ such that
\begin{equation}\label{EQ:masscon-h}
%\begin{cases}
\left\{
\begin{array}{cl}
u_h \rightarrow u, & \quad \mbox{as } h\to 0,\\
q_{n,h} \rightarrow q_n, & \quad \mbox{as } h\to 0, \\
\int_{\partial T} q_{n,h} ds = \int_T f dx, & \quad \forall T\in\T_h,
\end{array}
%\end{cases}
\right.
\end{equation}
where the convergence in \eqref{EQ:masscon-h} should be understood under certain prescribed topologies for the corresponding variables. The third line of \eqref{EQ:masscon-h} is the local mass conservation which is a highly preferred property of the algorithm in practical computing.

The classical continuous Galerkin finite element method is a popular and easy-to-implement numerical scheme for the model equation \eqref{ellipticbdy}-\ref{ellipticbc}. In the most simple formulation, the $P_1$-continuous Galerkin finite element scheme seeks $u_h\in S_h$ satisfying
\begin{equation}\label{EQ:CGFEM}
(\alpha\nabla u_h + \bbeta u_h, \nabla v) = (f,v),\qquad \forall v\in S_h,
\end{equation}
where $S_h\subset H_0^1(\Omega)$ consists of $C^0$-piecewise linear functions on a prescribed finite element triangulation $\T_h$. The numerical scheme (\ref{EQ:CGFEM}) provides a direct approximation $u_h$ to the primal variable $u=u(x)$ from which a numerical velocity can be computed as
\begin{equation}\label{EQ:naiveflux}
\bq_h = - ( \alpha \nabla u_h + \bbeta u_h)
\end{equation}
on each element $T\in \T_h$. It is not hard to see that the numerical velocity $\bq_h$ is discontinuous across the edges of $T\in\T_h$ along the normal direction (which is the flux). Furthermore, a simple average of the numerical flux $\bq_h$ on each edge may give a continuous numerical flux, but usually do not preserve the mass conservation property. For this reason, the continuous Galerkin finite element method is often said to be non-conservative with respect to $\T_h$.

Attempts of making the continuous Galerkin to be conservative have been made by various researchers in the scientific community for the last three decades. To the author's knowledge, all the existing work in this endeavor explore the use of various postprocessing techniques for the primal variable $u_h$. For example, in \cite{s-chou}, a conservative flux was obtained through a postprocessing procedure for either $P_1$-conforming or nonconforming Galerkin finite element approximations, where the reconstructed numerical flux was sought in the Raviart-Thomas space of the lowest order. In \cite{T-Hughes}, the authors devised a postprocessing procedure for continuous Galerkin finite element approximations on any user-selected subdomain. Specifically, for the subdomain under consideration, the authors introduced an auxiliary boundary flux field and developed a formulation which reduces to the usual continuous Galerkin method plus a modification designed for attaining global conservation. In \cite{Larson}, the authors developed a different postprocessing technique for computing a numerical flux on element boundaries that is element-wise conservative for the continuous Galerkin approximation. Their technique was based on the computation of a correction of the averaged normal flux on element boundaries by using the jump of piecewise constants or linear functions to be determined as the solution of a global linear system. A modified version of \cite{Larson} was employed in \cite{Odsater} to produce a conservative flux for both steady-state and dynamic flow models by adding a piecewise constant correction that is minimized in a weighted $L^2$ norm. The postprocessed flux in \cite{Odsater} was shown to have the same rate of convergence as the original, but non-conservative flux. In \cite{Bochev}, the authors studied the compatible least-squares method for the Darcy flow equation, and further developed a flux-correction procedure to obtain a locally conservative numerical solution without compromising its $L^2$ accuracy. In \cite{Cockburn_Jay}, the authors developed a two-step postprocessing procedure for computing a conservative flux for continuous Galerkin finite element approximation. Their first step involves the computation of a numerical flux trace defined on element interfaces. The second step is a local element-by-element postprocessing of the continuous Galerkin approximation by incorporating the result from the first step. In \cite{Mudunuro}, the authors devised a computational framework for advective-diffusive-reactive systems with approximate solutions satisfying desired properties such as maximum principles, the solution non-negative constraint, and element-wise conservative property. This method employs a low-order mixed finite element formulation based on least-squares formalism by enforcing explicit constraints of various type. In \cite{ZhangZhangZou_NMPDE_2017}, the authors develop
elementwisely-conservative  flux also by post-processing the FEM solution element-by-element. Their method is  valid for any order schemes and their post-processed solution converges with optimal order both in $H^1$ and $L^2$ norms. Very recently in \cite{ZouGuoDeng_SINUM_2017}, a volume-wisely conservative flux field has been derived by post-processing the finite element solution. One important feature of their method is that their derived flux field is continuous even across the internal edges of the underlying mesh.

In the literature, one can find various numerical methods for \eqref{ellipticbdy}-\eqref{ellipticbc} that preserve the mass conservation property \eqref{EQ:masscon-h} locally on each element $T\in\T_h$. One of such methods is the finite volume method (FVM) widely used in scientific computing for problems in science and engineering, including fluid dynamics \cite{Barth.T;Ohlberger2004,Emonot1992,EymardGallouetHerbin2000,LeVeque2002,Li.R2000,
Nicolaides1995,Shu2003}.
Most algorithms in FVM enjoy the nice feature of algorithmic simplicity and computational efficiency, and some of the low order FVMs (e.g., $P_0$ and $P_1$ schemes) have been well studied for their mathematical convergence and stability \cite{Bank.R;Rose.D1987,Cai.Z1991,ChouLi2000,Hackbusch.W1989a,Li.R2000,Xu.J.Zou.Q2009}. It should also be noted that the high order and symmetric FVMs are generally challenging in theory and algorithmic design \cite{CaiDouglasPark2003,Chen2010,ChenWuXu2012,LinYangZou2015,ZhangZou_NM_2015}. In the finite element context, several conservative numerical schemes have been developed. The mixed finite element method \cite{RaviartThomas, bdm}, the discontinuous Galerkin finite element method \cite{Arnold_SIAMJNA_2002}, the hybridizable discontinuous Galerkin \cite{Cockburn_SIAMJNA_1998}, and weak Galerkin finite element methods \cite{WangYe_2013, wy3655, WangWang_2016} are a few of such examples that give numerical approximations with conservative numerical flux.

In this work, we do not intend to pursue the approaches along above mentioned directions.
Instead, we shall design a new conservative numerical scheme for \eqref{ellipticbdy}-\eqref{ellipticbc} via a conservation-constrained optimization approach by using the continuous finite element space $S_h$. More precisely, let $V_h$ be a discrete flux space defined on the edge set of the triangulation $\T_h$, and $r\in [1,\infty)$ be any given real number. Our numerical method seeks $u_h\in S_h$ and $q_{n,h}\in V_h$ that minimizes the flux-error function
\begin{equation}\label{EQ:functional-intro}
J_r(u_h,q_{n,h}):= \frac{1}{r}\sum_{T\in\T_h} \sum_{e\subset\pT} h_T \int_e |q_{n,h}+\alpha\nabla u_h \cdot \bn +\bbeta u_h\cdot\bn|^r ds,
\end{equation}
subject to the constraint of the mass conservation equation of $\int_T q_{n,h} ds = \int_T f dx$ on each element $T\in \T_h$. In the case of $r=2$, the Euler-Lagrange form of this constrained optimization problem yields a system of linear equations for the unknown variables $u_h$, $q_{n,h}$, and another variable known as the Lagrange multiplier. Our new method essentially looks for a conservative flux variable that best approximates the obvious, but non-conservative numerical velocity $\bq_h = - (\alpha\nabla u_h +\bbeta u_h)$ in a discrete metric. For this reason, the numerical scheme is named as {\em Conservative Flux Optimization (CFO)} finite element method in this article. It should be pointed out that the {\em CFO} finite element method was originally motivated by the idea of the primal-dual weak Galerkin method (namely, PDE-constraint minimization of stabilizers) presented as in \cite{WangWang_2016} for the second order elliptic equation in non-divergence form.

The main contributions of this work are the following: (1) A conservative numerical scheme is devised to yield approximate solutions for the primal and the flux variables simultaneously without using any postprocessing. This flux-optimization based discretization technique is generally applicable to other partial differential equations; (2) An optimal order of convergence is established for the numerical flux in the $L^2$ norm; (3) The numerical solution $u_h$ for the primal variable is shown to have optimal order of convergence in $H^1(\Omega)$; (4) Numerical experiments are conducted on test problems involving varying and discontinuous coefficients, and the results strongly suggest an optimal order of convergence in $L^2(\Omega)$ for the primal variable; (5) The new scheme is applied to a simplified two-phase flow problem in a highly heterogeneous porous media, and the expected fingering phenomenon is clearly shown in the corresponding computational simulation.

 Our numerical scheme can also be viewed as a finite volume method -- vertex-based for the primal variable $u_h$ and element-based for the flux variable $q_{n,h}$.
In the literature, FVM algorithms are often classified into two categories: (1) vertex-centered where the computational nodes are positioned at the vertices, and (2) cell-centered where the computational nodes are positioned at the center of the cells/elements. Usually for vertex-centered schemes, one discretizes the underlying PDE by asking the numerical solution satisfying the flux conservation on a {\em dual mesh} consisting of {\it control volumes} given as polygons (2D) or polyhedra (3D) surrounding the mesh vertices. The mathematical convergence for vertex-centered schemes is often established in a way that mimics the corresponding finite element formulation \cite{Bank.R;Rose.D1987,Cai.Z1991,Hackbusch.W1989a,Schmidt1993,Xu.J.Zou.Q2009,ZhangZou_NM_2015}.
For cell-centered finite volume schemes, their theoretical convergence is usually obtained in carefully chosen discrete norms on meshes with certain structures \cite{EymardGallouet2006,Lazarov1996,Suli_MCOM_1992,Yubo2012}.
The new scheme of this paper belongs to the category of element-centered FVMs in a broad sense, as the mass conservation \eqref{EQ:conservation} is imposed on each element of the partition $\T_h$. Compared with the vertex-centered FVMs, our scheme does not require a deliberate design of any dual mesh in the computation. Compared with the cell-centered FVMs, our numerical solution $u_h$ belongs to the continuous finite element space so that its convergence is within the reach of currently available mathematical techniques with generalization to unstructured partitions. It worths mentioning that our scheme is symmetric for non-symmetric PDE problems while the classical vertex-centered FVMs are often non-symmetric even for symmetric PDE problems.

The paper is organized as follows. In Section \ref{sectionccfv}, we present the conservative flux optimization finite element scheme for the model problem \eqref{ellipticbdy}-\eqref{ellipticbc}. In Section \ref{sectionWpS}, we establish a result on the well-posedness and stability for the conservative flux optimization scheme. In Section \ref{sectionEE}, we derive some error estimates for the resulting numerical approximations in various Sobolev norms. Finally, in Section \ref{numerical-experiments}, we present a few numerical results to demonstrate the efficiency and accuracy of the new scheme. In particular, we will first verify the theoretical convergence through couple of testing examples, and then demonstrate the power of the new scheme in scientific computing through a simplified two-phase flow problem in highly heterogeneous porous media.

\section{A Conservative Flux Optimization Scheme}\label{sectionccfv}

Let $\T_h$ be a regular triangulation of the polygonal domain $\Omega\subset\mathbb{R}^2$. Denote by $h\coloneqq \max_{{T} \in \T_h}h_{T}$ the meshsize of $\T_h$, where $h_{T}=\diam({T})$ is the diameter of the element ${T} \in \T_h$. Denote by $\E_h$ the edge set of $\T_h$, and $\E_h^0 \subset \E_h$ the set of all interior edges. The set of boundary edges is denoted as $\E_h^B\coloneqq\E_h \backslash\E_h^0$. The diameter of the edge $e\in \E_h$ is denoted as $h_e=\diam(e)$. For convenience, for each $e\in\E_h$ we assign a normal direction $\bn_e$ which provides an orientation for $E_h$.

Recall that the classical ${P}_1$-conforming element is given by
\[
S_h=\{v\in C^0(\Omega):  \  v|_{T}\in {P}_1(T), \forall {T}\in\T_h, v|_{\partial\Omega}=0\}.
\]
Here $C^0(\Omega)$ stands for the space of continuous functions on the domain $\Omega$. Denote by $W_h$ the space of piecewise constant functions on $\T_h$ and $V_h$ the space of piecewise constant functions on the edge set $\E_h$.

%For simplicity, we denote $\U_h=U_h\times V_h\times W_h$, $\V_h=U_h\times V_h$ and $\W_h=V_h\times W_h$.\\

For any $q\in L^2(\E_h)$, denoted by $\nabla_w\cdot q$ the discrete weak divergence given as a function in $W_h$ such that on each $T\in\T_h$
\begin{equation}\label{EQ:001}
(\nabla_w\cdot q)|_T =\frac{1}{|T|}\int_{\partial T} q \bn\cdot {\bf n}_e ds,
\end{equation}
where $\bn$ is the outward normal vector of $\partial T$, and $\bn_e$ is the prescribed orientation of $e\subset\partial T$; see Figure \ref{fig:triangle} for an illustration where $\bn$ coincides with $\bn_e$.

\begin{figure}
\begin{center}
\begin{tikzpicture}[rotate=233]
    %define the points of triangle
    %you can also use \coordinate to define these points, display in the following case
    \path (0,0) coordinate (A3);
    \path (3,0) coordinate (A1);
    \path (3.4,-0.2) coordinate (A11);
    \path  (1.92, 1.44) coordinate (M);
    \path (0,4) coordinate (A2);
    \path (-0.2,4.4) coordinate (A21);
    \path (1.0,1.0) coordinate (center);
    \path (1.5,0) coordinate (A1half);
    \path (0,2) coordinate (A2half);
    \path (1.5,2) coordinate (A3half);
    \path (2.0,0.5) coordinate (a);
    \path (A1half) ++(0,-0.6) coordinate (A1To);
    \path (A2half) ++(-0.6,0) coordinate (A2To);
    \path (A3half) ++(38:0.6cm) coordinate (A3To);
    \draw (A3) -- (A1) -- (A2) -- (A3);
    %\draw [dashed] (A3) -- (M);
    %\filldraw[black] (A1) circle(0.1);
    %\filldraw[black] (A2) circle(0.1);
    %\filldraw[black] (A3) circle(0.1);
    %\filldraw[black] (A1half) circle(0.1);
    %\filldraw[black] (A2half) circle(0.1);
    %\filldraw[black] (A3half) circle(0.1);
    %\filldraw[black] (center) circle(0.11);
    %\draw node[above] at (A3)(2.625, 0.5) {$V_3$};
    %\draw node[below] at (A1) {$V_1$};
    %\draw node[below] at (M) {$M$};
    %\draw node[below] at (A11) {$A_1$};
    %\draw node[below] at (A2) {$V_2$};
    \draw node at (center) {$P_1(T)$};
    \draw node[right] at (A1half)(2.625, 0.5) {$q_{2}$};
    %\draw node[below] at (A1) {A(1)};
    \draw node[left] at (A2half) {$q_{1}$};
    \draw node[above] at (A3half) {$q_{3}$};
    %\draw node[below] at (center) {$P_{j}(T)$};

    \draw[->,thick] (A1half) -- (A1To) node[above]{$\mathbf{n}_2$};
    \draw[->,thick] (A2half) -- (A2To) node[right]{$\mathbf{n}_1$};
    \draw[->,thick] (A3half) -- (A3To) node[right]{$\mathbf{n}_3$};
    %\draw (2.625, 0.5) arc(135:180:0.725);
    %\draw node at (a) {$\alpha_{23}$};
\end{tikzpicture}
\end{center}
\caption{An illustrative triangular element with local flux}
\label{fig:triangle}
\end{figure}
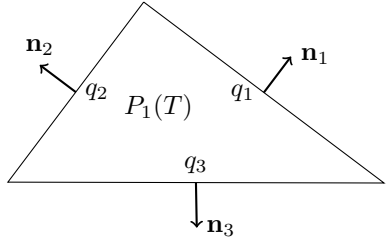

A flux function $p\in V_h$ is said to be an {\it admissible discrete flux} to the equation $\nabla\cdot\bq = f$ if it satisfies
\begin{equation}\label{fv}
(\nabla_w\cdot p, w)=(f, w),\qquad \forall w\in W_h.
\end{equation}
As $W_h$ consists of piecewise constant functions, the equation (\ref{fv}) is equivalent to the following element-wise identity
\begin{equation}\label{fv-ele}
(\nabla_w\cdot p, 1)_T=(f, 1)_T,\qquad \forall T\in \T_h
\end{equation}
Let $Q_h$ be the $L^2$ projection onto the finite element space $W_h$. From (\ref{fv}), a discrete flux $p\in V_h$ is {\em admissible} if and only if $\nabla_w \cdot p = Q_h f$.

The convection-diffusion equation (\ref{ellipticbdy}) can be rewritten as a system of linear equations:
\begin{equation}\label{EQ:system-linear-eq}
\bq = - (\alpha \nabla u + \bbeta u),\qquad \nabla\cdot\bq = f.
\end{equation}
A pair of finite element functions $(u_h;q_h)\in S_h\times V_h$ is said to be an ideal approximation of (\ref{EQ:system-linear-eq}) with the boundary condition (\ref{ellipticbc}) if, on each element $T\in\T_h$, one has
\begin{eqnarray}\label{EQ:ideal-solution-1}
q_h + (\alpha^* \nabla u_h + \bbeta^* u_h)&=&0\qquad \mbox{ on } \pT,\\
\nabla_w\cdot q_h & = & Q_h f\qquad \mbox{ in } T, \label{EQ:ideal-solution-2}
\end{eqnarray}
where $\alpha^*$ and $\bbeta^*$ are the trace of $\alpha$ and $\bbeta$ on $\pT$ as taken from the element $T$, respectively. The equation (\ref{EQ:ideal-solution-2}) ensures the local mass conservation by the discrete flux $q_h$.

It is not hard to see that the ideal approximation for (\ref{EQ:system-linear-eq}) generally may not exist in the finite element space $S_h\times V_h$. A remedy to the solution nonexistence is to find a pair $(u_h;q_h)\in S_h\times V_h$ that satisfies the mass conservation equation (\ref{EQ:ideal-solution-2}) while the flux error $q_h + (\alpha^* \nabla u_h + \bbeta^* u_h)$ be minimized in a metric at the user's discretion. To this end, we introduce a functional in the space $S_h\times V_h$ as follows:
\begin{equation}\label{EQ:functional}
J_r(v,p):= \frac{1}{r}\sum_{T\in\T_h} \sum_{e\subset\pT} h_T \int_e |p+\alpha^*\nabla v \cdot \bn_e +\bbeta^* v\cdot\bn_e|^r ds,
\end{equation}
where $r\in [1,\infty)$ is a prescribed value. Our numerical algorithm then seeks $(u_h;q_h)\in S_h\times V_h$ which minimizes the functional $J_r$ under the constraint (\ref{EQ:ideal-solution-2}). Due to the emphasis on the mass conservation and the error minimization for the flux approximation, we shall name this discretization algorithm a {\em Conservative Flux Optimization Finite Element Method} or {\em Conservative Flux Optimization (CFO)} in brief. The CFO algorithm can be mathematically stated as follows:

\begin{algorithm}
Find $u_h\in V_h$ and $q_h\in U_h$ such that
\begin{equation}\label{min}
(u_h;q_h)=\argmin_{v\in S_h, p\in V_h, \nabla_w \cdot p = Q_h f} J_r(v,p).
\end{equation}
\end{algorithm}

\bigskip

By introducing a Lagrange multiplier $\lambda_h\in M_h$, the constrained minimization problem (\ref{min}) can be reformulated in the Euler-Lagrange form that seeks $(u_h;q_h)\in S_h\times V_h$ and $\lambda_h\in M_h$ satisfying
\begin{eqnarray}\label{min-LagrangeForm-01}
\langle D J_r(u_h,q_h), (v;p)\rangle + (\nabla_w\cdot p, \lambda_h)&=& 0,\qquad \forall (v;p)\in S_h\times V_h,\\
(\nabla_w\cdot q_h, w) & = & (f,w),\qquad \forall w\in W_h, \label{min-LagrangeForm-02}
\end{eqnarray}
where
\begin{equation}\label{EQ:FDerivative}
\langle D J_r(u_h,q_h), (v;p)\rangle:= \sum_{T\in\T_h} h_T\sum_{e\subset\pT} \int_e |q_h+\alpha^*\nabla u_h \cdot \bn_e +\bbeta^* u_h\cdot\bn_e|^{r-1} (p+\alpha^*\nabla v \cdot \bn_e +\bbeta^* v\cdot\bn_e )\ Sgn \ ds
\end{equation}
is the Fr\'echet derivative of the functional $J_r$ at $(u_h;q_h)$ along the direction of $(v;p)$. Here $Sgn=\sgn(q_h+\alpha^*\nabla u_h \cdot \bn_e +\bbeta^* u_h\cdot\bn_e)$ represents the sign of the corresponding term. For the case of $r=2$, the Fr\'echet derivative $\langle D J_r(u_h,q_h), (v;p)\rangle$ defines a bilinear form given as follows
\begin{equation}\label{EQ:FDerivative2}
s_h(u_h;q_h), (v;p):= \int_{T\in\T_h} \sum_{e\subset\pT} h_T\int_e (q_h+\alpha^*\nabla u_h \cdot \bn_e +\bbeta^* u_h\cdot\bn_e) (p+\alpha^*\nabla v \cdot \bn_e +\bbeta^* v\cdot\bn_e )\ ds,
\end{equation}
so that the Euler-Lagrange equations (\ref{min-LagrangeForm-01})-(\ref{min-LagrangeForm-02}) read as below
\begin{eqnarray}\label{min-LagrangeForm2-01}
s_h(u_h;q_h), (v;p)) + (\nabla_w\cdot p, \lambda_h)&=& 0,\qquad \forall (v;p)\in S_h\times V_h,\\
(\nabla_w\cdot q_h, w) & = & (f,w),\qquad \forall w\in W_h. \label{min-LagrangeForm2-02}
\end{eqnarray}

\section{Well-posedness and Stability}\label{sectionWpS}

In this section, we shall study the solution existence and uniqueness of the conservative flux optimization finite element scheme (\ref{min}) with $r=2$. Recall that the corresponding Euler-Lagrange formulation is a system of linear equations given by (\ref{min-LagrangeForm2-01})-(\ref{min-LagrangeForm2-02}).

\begin{definition} For any element $T\in \T_h$, the oscillation of a function $\sigma=\sigma(x)\in L^\infty(\Omega)$ on $T$ is defined as
$$
\osc(\sigma, T):= \essensup_{x,y\in T} |\sigma(x)-\sigma(y)|
$$
The local oscillation of $\sigma=\sigma(x)$ with respect to the partition $\T_h$ is given by
$$
\osc(\sigma, \T_h):=\max_{T\in\T_h} \osc(\sigma, T).
$$
\end{definition}

\begin{definition}
For a given non-negative integer $m$, a function $\sigma=\sigma(x)\in L^\infty(\Omega)$ is said to be uniformly piecewise $C^m$ with respect to the partition $\T_h$ if $\sigma|_T \in W^{m,\infty}(T)$ for each $T\in \T_h$. Furthermore, for any given $\epsilon>0$, there exists a $\delta>0$ such that
$$
\osc(\partial^{s} \sigma, \T_h)< \epsilon,\qquad 0\le |s|=s_1+s_2 \le m
$$
whenever $h<\delta$. Here $s=(s_1,s_2)$ is a multi-index and $\partial^s = \partial_x^{s_1}\partial_y^{s_2}$ is the corresponding partial differential operator.
\end{definition}

Denote by $C^m(\T_h)$ the space of functions that are uniformly piecewise $C^m$ with respect to the finite element partition $\T_h$.

Next, for any given $\theta\in H^{-1}(\Omega)$ let $w=w(\theta)\in H_0^1(\Omega)$ be the solution of the following auxiliary problem:
\begin{equation}\label{EQ:auxiliary-problem}
(\nabla w, \alpha\nabla\phi + \bbeta \phi)= (\theta, \phi),\qquad \forall \phi\in H_0^1(\Omega).
\end{equation}
The problem (\ref{EQ:auxiliary-problem}) is given by the dual of the differential operator in the model problem (\ref{ellipticbdy}). From the solution existence and uniqueness assumption for (\ref{ellipticbdy}), the problem (\ref{EQ:auxiliary-problem}) has a unique solution $w\in H_0^1(\Omega)$ satisfying the a priori estimate
\begin{equation}\label{EQ:August:001}
\|w\|_1 \leq C \|\theta\|_{-1}.
\end{equation}

\begin{lemma}\label{Lemma:GeneralIdentity}
Assume that the coefficients of the model PDE (\ref{ellipticbdy}) are uniformly piecewise continuous; i.e., $\alpha\in C^0(\T_h)$ and $\bbeta\in  C^0(\T_h)$. Given any $v\in S_h$, let $q_v = \alpha\nabla v +\beta v$ be the corresponding flux. Then for any $\theta\in H^{-1}(\Omega)$, the following identity holds true
\begin{eqnarray}\label{erreqn1}
(\theta, v)&=&\sum_{T\in\T_h}\sum_{e\subset\partial T} \langle(q_v + p \bn_e) \cdot \bn, w-Q_0w\rangle_e  +(\nabla_w \cdot p, w) \nonumber \\
          &&+\sum_{T\in\T_h} ((I-Q_0)q_v,\nabla w)_T + \sum_{T\in\T_h}\sum_{e\subset\partial T} \langle(Q_0-I)q_v\cdot \bn, w-Q_0w\rangle_e,
\end{eqnarray}
where $p\in V_h$ is arbitrary, $Q_0$ is the $L^2$ projection onto $W_h$, and
$w\in H^1_0(\Omega)$ is the solution of the auxiliary problem (\ref{EQ:auxiliary-problem}).
 \end{lemma}

\begin{proof}
As $v\in S_h\subset H_0^1(\Omega)$, we have from (\ref{EQ:auxiliary-problem})
\begin{equation}\label{EQ:March08:001}
\begin{split}
(\theta,v)=&(\alpha \nabla v +\bbeta v,\nabla w)\\
=&\sum_{T\in\T_h} (Q_0q_v,\nabla w)_T+\sum_{T\in\T_h} ((I-Q_0)q_v,\nabla w)_T\\
=&\sum_{T\in\T_h}\sum_{e\subset\partial T} \langle Q_0q_v\cdot \bn, w\rangle_e +\sum_{T\in\T_h} ((I-Q_0) q_v,\nabla w)_T\\
=&\sum_{T\in\T_h}\sum_{e\subset\partial T} \langle(Q_0 q_v+p \bn_e) \cdot \bn, w\rangle_e +\sum_{T\in\T_h} ((I-Q_0)q_v,\nabla w)_T,
\end{split}
\end{equation}
where we have used the fact that $\sum_{T\in\T_h}\sum_{e\subset\partial T} \langle p \bn_e\cdot \bn, w\rangle_e=0$ for any $p\in V_h$
in the last equality.
From $\int_{\partial T} Q_0 q_v\cdot \bn ds=0$  we arrive at
\begin{eqnarray*}
\sum_{T\in\T_h}\sum_{e\subset\partial T} \langle(Q_0q_v+p \bn_e) \cdot \bn, Q_0w\rangle_e = (\nabla_w \cdot p, Q_0 w)=(\nabla_w \cdot p, w).
\end{eqnarray*}
Using the above identity in (\ref{EQ:March08:001}) we obtain
\begin{equation}\label{EQ:March08:002}
\begin{split}
(\theta, v)=&\sum_{T\in\T_h}\sum_{e\subset\partial T} \langle(Q_0 q_v+p\bn_e) \cdot \bn, w-Q_0w\rangle_e +\sum_{T\in\T_h} ((I-Q_0)q_v,\nabla w)_T+(\nabla_w\cdot p, w)\\
=&\sum_{T\in\T_h}\sum_{e\subset\partial T} \langle(q_v+p \bn_e) \cdot \bn, w-Q_0w\rangle_e +\sum_{T\in\T_h} ((I-Q_0)q_v,\nabla w)_T\\
&\quad + \sum_{T\in\T_h}\sum_{e\subset\partial T} \langle (Q_0-I)q_v\cdot \bn, w-Q_0w\rangle_e+(\nabla_w\cdot p, w),
\end{split}
\end{equation}
which gives the identity (\ref{erreqn1}).
\end{proof}

\medskip

In the space $S_h\times V_h$, we introduce the following semi-norm
\begin{eqnarray}\label{norm}
\|(v,p)\|_h&=&\left(J_2(v,p)+\|\nabla_w\cdot p\|_{-1}^2\right)^{\frac12},\qquad (v,p)\in S_h\times V_h.
\end{eqnarray}
An application of Lemma \ref{Lemma:GeneralIdentity} shows that $\|(v,p)\|_h$ indeed defines a norm in the discrete space $S_h\times V_h$. More precisely, we have the following result.

\begin{lemma}\label{Lemma:H1NormBound}
Assume that the coefficients $\alpha$ and $\bbeta$ are in $C^0(\T_h)$.  Then for any $v\in S_h$ and $p\in V_h$, we have
\begin{equation}\label{EQ:March09:101}
\|v\|_1 \leq C\left((J_2(v, p)^{1/2}+\|\nabla_w\cdot p\|_{-1}\right),
\end{equation}
provided that the meshsize $h$ is sufficiently small. Consequently, the semi-norm $\|\cdot\|_h$ becomes to be a norm in the space $S_h\times V_h$ when the meshsize is sufficiently small.
\end{lemma}

\begin{proof}
Using the identity (\ref{erreqn1}) we have
\begin{equation}\label{EQ:March09:001}
\begin{split}
|(\theta, v)|\leq & \sum_{T\in\T_h}\sum_{e\subset\partial T} |\langle(q_v + p \bn_e) \cdot \bn, w-Q_0w\rangle_e| +\|\nabla_d\cdot p\|_{-1}\|w\|_1\\
 & +\sum_{T\in\T_h} |((I-Q_0)q_v,\nabla w))_T| + \sum_{T\in\T_h}\sum_{e\subset\partial T} |\langle(Q_0-I)q_v\cdot \bn, w-Q_0w\rangle_e |\\
\leq & \sum_{T\in\T_h}\sum_{e\subset\partial T} \|(q_v + p \bn_e) \cdot \bn\|_e\| w-Q_0w\|_e +\|\nabla_d\cdot p\|_{-1}\|w\|_1\\
 & +\sum_{T\in\T_h} \|(I-Q_0)q_v\|_T \|\nabla w\|_T + \sum_{T\in\T_h}\sum_{e\subset\partial T} \|(Q_0-I)q_v\cdot \bn\|_e \|w-Q_0w\|_e.
\end{split}
\end{equation}

Each term on the right-hand side of (\ref{EQ:March09:001}) can be handled respectively as follows. From the Cauchy-Schwarz inequality and the trace inequality, we may estimate the first term by
\begin{equation}\label{EQ:August22:001}
\begin{split}
&\sum_{T\in\T_h}\sum_{e\subset\partial T} \|(\alpha\nabla v + p \bn_e) \cdot \bn\|_{e} \|w-Q_0w\|_e\\
\leq & C\left(\sum_{T\in\T_h} h_T\|(\alpha\nabla v + p \bn_e) \cdot \bn\|_{\partial T}^2\right)^{1/2} (h_T^{-1}\|w-Q_0w\|_{0,T} +\|\nabla w\|_{0,T})\\
\leq & C J_2(v,p)^{1/2} \|\nabla w\|_0.
\end{split}
\end{equation}

Next, observe that, on each element $T$, $(I-Q_0) q_v$ may be rewritten as
\begin{equation}\label{EQ:qv-decomposition}
(I-Q_0) q_v = (I-Q_0) \left( (\alpha - \overline{\alpha})\nabla v
+ \bbeta v - \overline{\bbeta}\overline{v}
\right),
\end{equation}
where $\overline{\sigma}$ stands for the cell average of the underlying function $\sigma$ over the element $T$. It follows that
\begin{equation*}
  \begin{split}
\|(I-Q_0) q_v\|_{0,T} & \leq \| (\alpha - \overline{\alpha})\nabla v +
\bbeta v - \overline{\bbeta}\overline{v}\|_{0,T}\\
& \leq \| (\alpha - \overline{\alpha})\nabla v\|_{0,T} +
\|\overline{\bbeta} (v-\overline{v})\|_{0,T} + \|(\bbeta -\overline{\bbeta}) v\|_{0,T}\\
& \leq \osc(\alpha, T) \|\nabla v\|_{0,T} + \osc(\bbeta, T)\|v\|_{0,T}
+ C h_T \|\nabla v\|_{0,T}\\
& \leq (\osc(\alpha,T) + \osc(\bbeta,T) + Ch_T)\left(\|\nabla v\|_{0,T} + \|v\|_{0,T}\right),
\end{split}
\end{equation*}
Thus, the third term has the following estimate:
\begin{equation}\label{EQ:August22:002}
\begin{split}
\sum_{T\in\T_h} \|(I-Q_0)q_v\|_T \|\nabla w\|_T  \leq & \sum_{T\in\T_h} (\osc(\alpha,T) +\osc(\bbeta, T)+Ch_T)\left(\|\nabla v\|_{0,T} + \|v\|_{0,T}\right)\|\nabla w\|_T \\
\leq & (\osc(\alpha, \T_h)+\osc(\bbeta,\T_h)+Ch) \|v\|_1 \|w\|_1.
\end{split}
\end{equation}

As for the fourth term, we again use the decomposition (\ref{EQ:qv-decomposition}) to arrive at
\begin{equation}\label{EQ:August22:004}
\begin{split}
& \sum_{T\in\T_h}\sum_{e\subset\partial T} \|(Q_0-I)q_v\cdot \bn\|_e \|w-Q_0w\|_e\\
\leq & \sum_{T\in\T_h}\sum_{e\subset\partial T} \|(I-Q_0) \left( (\alpha - \overline{\alpha})\nabla v
+ \bbeta v - \overline{\bbeta}\overline{v}
\right)\|_e \|w-Q_0w\|_e.
\end{split}
\end{equation}
For any function $\sigma$ defined on $T\in \T_h$, from the trace and inverse inequality we have
$$
\|Q_0\sigma\|_e \leq C (h^{-1} \|Q_0\sigma\|_{0,T}^2 + h \|\nabla(Q_0\sigma)\|_{0,T}^2) \leq C h^{-1} \|Q_0\sigma\|_{0,T}^2.
$$
It follows that
\begin{eqnarray*}
& &\|(I-Q_0) \left( (\alpha - \overline{\alpha})\nabla v
+ \bbeta v - \overline{\bbeta}\overline{v}
\right)\|_e \\
&\leq & \|(\alpha - \overline{\alpha})\nabla v
+ \bbeta v - \overline{\bbeta}\overline{v}\|_e + \|Q_0\left( (\alpha - \overline{\alpha})\nabla v
+ \bbeta v - \overline{\bbeta}\overline{v}\right)\|_e\\
& \leq & \|(\alpha - \overline{\alpha})\nabla v
+ \bbeta v - \overline{\bbeta}\overline{v}\|_e + Ch^{-1/2} \|Q_0 \left( (\alpha - \overline{\alpha})\nabla v
+ \bbeta v - \overline{\bbeta}\overline{v}\right)\|_{0,T}\\
& \leq & \|(\alpha - \overline{\alpha})\nabla v
+ \bbeta v - \overline{\bbeta}\overline{v}\|_e + Ch^{-1/2} \|(\alpha - \overline{\alpha})\nabla v
+ \bbeta v - \overline{\bbeta}\overline{v}\|_{0,T}\\
& \leq & (\osc(\alpha,T)+\osc(\bbeta,T)+Ch_T) (\|\nabla v\|_e + \|v\|_e)
 + Ch^{-1/2} \|(\alpha - \overline{\alpha})\nabla v
+ \bbeta v - \overline{\bbeta}\overline{v}\|_{0,T}\\
& \leq & C(\osc(\alpha,T)+\osc(\bbeta,T)+h_T)h^{-1/2} (\|\nabla v\|_T + \|v\|_T).
\end{eqnarray*}
Substituting the above estimate into (\ref{EQ:August22:004}) yields
\begin{equation}\label{EQ:August22:006}
\begin{split}
& \sum_{T\in\T_h}\sum_{e\subset\partial T} \|(Q_0-I)q_v\cdot \bn\|_e \|w-Q_0w\|_e\\
\leq & C \sum_{T\in\T_h} (\osc(\alpha,T)+\osc(\bbeta,T)+h_T)h^{-1/2} (\|\nabla v\|_T + \|v\|_T)
\|w-Q_0w\|_\pT\\
\leq & C (\osc(\alpha,\T_h)+\osc(\bbeta, \T_h)+h) \|v\|_1 \|w\|_1.
\end{split}
\end{equation}

Now by combining (\ref{EQ:March09:001}) with (\ref{EQ:August22:001}),  (\ref{EQ:August22:002}), and (\ref{EQ:August22:006}) we arrive at
\begin{equation}
\begin{split}
(\theta, v) & \leq C J_2(v,p)^{1/2} \|\nabla w\|_0 + \|\nabla_w\cdot p\|_{-1}\|w\|_1 + C (\osc(\alpha,\T_h)+\osc(\bbeta,\T_h)+h) \|v\|_1 \|w\|_1\\
&\leq C (J_2(v,p)^{1/2} + \|\nabla_w\cdot p\|_{-1} + (\osc(\alpha,\T_h)+\osc(\bbeta,\T_h)+h) \|v\|_1)\|\theta\|_{-1}.
\end{split}
\end{equation}
It follows that
$$
\|v\|_1 \leq C \left(J_2(v,p)^{1/2} + \|\nabla_w\cdot p\|_{-1} + (\osc(\alpha,\T_h)+\osc(\bbeta,\T_h)+h) \|v\|_1\right),
$$
which, with the assumption of $\alpha\in C^0(\T_h)$ and $\bbeta\in C^0(\T_h)$, implies the estimate (\ref{EQ:March09:101}) for sufficiently small meshsize $h$.

To show that $\|\cdot\|_h$ defines a norm in $S_h\times V_h$, for any $(v,p)\in S_h\times V_h$ satisfying $\|(v,p)\|_{h}=0$ we have $J_2(v,p)=0$ and $\nabla_w\cdot p=0$. It then follows from (\ref{EQ:March09:101}) that $v=0$. From $S(v,p)=0$ and $v=0$, we further have $p=p+\alpha \nabla v =0$ on each edge $e\in \E_h$. This completes the proof of the lemma.
\end{proof}

\medskip

In the space $W_h$, we introduce a discrete $H^1$ norm as follows
$$
\|\sigma\|_{1,h}=\left( \sum_{e\in\E_h}  [\![ \sigma]\!]_e^2 \right)^{1/2},
$$
where $[\![ \sigma]\!]_e = \sigma|_{T_L} - \sigma|_{T_R}$ is the jump of the piecewise constant function $\sigma$ on the edge $e\in \E_h$ shared by two elements $T_L$ and $T_R$. For $e$ on the boundary $\partial\Omega$, we assume $T_R$ is empty so that a one-sided trace of $\sigma$ will be taken as the value of $[\![ \sigma]\!]_e$.

Next, we equip the space $V_h$ with the following $L^2$-norm:
$$
\|p\|_0=\left( \sum_{e\in\E_h} h_e\int_e p^2 ds \right)^{1/2}.
$$

\begin{lemma}\label{infsuplm}
For any $\sigma\in W_h$, there exists a discrete flux $p_\sigma\in V_h$ such that
$$
(\nabla_w\cdot p_\sigma,\sigma)=\|\sigma\|_{1,h}^2,\qquad \|p_\sigma\|_0 \leq C \|\sigma\|_{1,h}.
$$
\end{lemma}

\begin{proof} From the computational formula (\ref{EQ:001}) for $\nabla_w\cdot p$, we have
\begin{eqnarray*}
(\nabla_w\cdot p, \sigma) &=&\sum_T \int_{e\subset \partial T} p|_e \bn_e\cdot\bn\sigma ds \\
&=& \sum_{e\in\E_h} \int_e [\![ \sigma]\!]_e p|_e ds,
\end{eqnarray*}
By choosing $p_\sigma|_e =  [\![ \sigma]\!]_e / h_e$ we arrive at
$$
(\nabla_w\cdot p_\sigma),\sigma) = \sum_{e\in\E_h}  [\![ \sigma]\!]_e^2 = \|\sigma\|_{1,h}^2.
$$
Furthermore, it is easy to see that
$$
\|p_\sigma\|_0^2 = \sum_{e\in\E_h} h_e\int_e p_\sigma^2 ds = \sum_{e\in\E_h}
[\![ \sigma]\!]_e^2 = \|\sigma\|_{1,h}^2.
$$
This completes the proof of the lemma.
\end{proof}

\begin{theorem}
Assume that the coefficients $\alpha$ and $\bbeta$ are in $C^0(\T_h)$. The numerical scheme \eqref{min-LagrangeForm2-01}-\eqref{min-LagrangeForm2-02} has one and only one solution for $u_h$, $q_h$, and $\lambda_h$, provided that the meshsize $h$ is sufficiently small.
\end{theorem}

\begin{proof}
It suffices to show that the homogeneous problem has only trivial solution. To this end, let $u_h\in S_h$, $q_h\in V_h$, and $\lambda_h\in W_h$ be the solution of \eqref{min-LagrangeForm2-01}-\eqref{min-LagrangeForm2-02} with homogeneous data $f=0$. From the equation \eqref{min-LagrangeForm2-02} we have $\nabla_w \cdot q_h=0$ on each element $T\in\T_h$. Next, by choosing $(v,p)=(u_h, q_h)$ in \eqref{min-LagrangeForm2-01} and using $(\nabla_w\cdot q_h,\lambda_h)=0$ we obtain
$$
J_2(u_h,q_h)=0,
$$
which, together with $\nabla_w\cdot q_h=0$, leads to $\|(u_h,q_h)\|_h=0$ by using (\ref{norm}). It follows that $u_h=0$ and $q_h=0$. Thus, from \eqref{min-LagrangeForm2-01} we obtain
$$
(\nabla_w\cdot p, \lambda_h) = 0,\qquad \forall p\in V_h.
$$
Next, from the {\em inf-sup} result of Lemma \ref{infsuplm} there exists a discrete flux $p_{\lambda_h}\in V_h$ satisfying
$$
\|\lambda_h\|_{1,h}^2 = (\nabla_w\cdot p_{\lambda_h}, \lambda_h)=0,
$$
which gives rise to $[\![\lambda_h]\!]=0$ on each edge $e\in \E_h$ (including the boundary edge), and hence $\lambda_h\equiv 0$. This completes the proof of the theorem.
\end{proof}

\section{Error Estimates}\label{sectionEE}
We are now in a position to establish some error estimates for the approximate solutions $u_h$ and $q_h$ arising from the scheme (\ref{min-LagrangeForm2-01})-(\ref{min-LagrangeForm2-02}).

\begin{theorem}\label{thm:h1}
Assume that the coefficients of the model problem (\ref{ellipticbdy}) are uniformly piecewise continuous; i.e., $\alpha\in C^0(\T_h)$ and $\bbeta\in C^0(\T_h)$. Let $u$ be the exact solutions of (\ref{ellipticbdy})-(\ref{ellipticbc}), $u_h\in S_h$, $q_h\in V_h$, and $\lambda_h\in W_h$ be the approximate solution arising from the numerical scheme (\ref{min-LagrangeForm2-01})-(\ref{min-LagrangeForm2-02}). Assume $u\in H^2(\Omega)$, then the following error estimate holds true
\begin{equation}\label{h1}
\|u-u_h\|_1\le C (\|\alpha\|_\infty h \|u\|_2 + \|\bbeta\|_\infty h \|u\|_1 + \osc(\alpha, \T_h) \|u\|_1 + \osc(\bbeta,\T_h) \|u\|_0).
\end{equation}
In particular, if the coefficients $\alpha$ and $\bbeta$ are differentiable on each element $T\in \T_h$, then
\begin{equation}\label{h1-02}
\|u-u_h\|_1\le C h \|u\|_2,
\end{equation}
where $C=C(\alpha,\bbeta)$ is a constant depending on the element-wise $W^{1,\infty}$ norm of the coefficients $\alpha$ and $\bbeta$.
\end{theorem}

\begin{proof}
Let $u_I$ be the usual nodal point interpolation of the exact solution $u$. Denote by $Q_b q$ the piecewise constant interpolation of the exact flux $q=-(\alpha\nabla u +\bbeta u)\cdot \bn_e$ in the finite element space $V_h$. Let the error functions be given by
$$
e_h = u_h - u_I,\quad \eta_h=q_h-Q_b q,\quad \xi_h=\lambda_h - 0.
$$
It is not hard to see that $\nabla_w\cdot \eta_h =0$, and from the estimate (\ref{EQ:March09:101}) we have
\begin{equation}\label{EQ:March01:102}
\|e_h\|_1 \leq C J_2(e_h, \eta_h)^{\frac{1}{2}}.
\end{equation}
Thus, it suffices to derive an estimate for $J_2(e_h, \epsilon_h)$. To this end, we observe that the triplet $(u_I, Q_b q, \lambda_I=0)$ satisfies
\begin{eqnarray}
s_h((u_I,Q_b q), (v,p))+ (\nabla_w\cdot p,\lambda_I) &=& s_h((u_I,Q_b q), (v,p)), \qquad \forall v\in S_h, p\in V_h,\label{Lagrangem11-new}\\
(\nabla_w\cdot (Q_b q),w) &=& (f,w),\qquad \forall w\in W_h.\label{Lagrangem12-new}
\end{eqnarray}
By subtracting (\ref{Lagrangem11-new}) from (\ref{min-LagrangeForm2-01}) and (\ref{Lagrangem12-new}) from (\ref{min-LagrangeForm2-02}), we arrive at the following error equations
\begin{eqnarray}
s_h((e_h,\eta_h), (v,p))+ (\nabla_w\cdot p,\xi_h) &=& -s_h((u_I,Q_b q), (v,p)),\qquad \forall v\in S_h, \ p\in V_h, \label{Lagrangem11-newnew}\\
(\nabla_w\cdot \eta, w) &=& 0\qquad \forall w\in W_h.\label{Lagrangem12-newnew}
\end{eqnarray}
By letting $(v,p)=(e_h, \eta_h)$ in (\ref{Lagrangem11-newnew}) we further obtain
$$
J_2(e_h,\eta_h) = - s_h((u_I,Q_b q), (e_h,\eta_h)).
$$
Using the Cauchy-Schwarz inequality to the right-hand side of the above equality gives
\begin{equation}\label{0325-01}
J_2(e_h,\eta_h) \leq  J_2(u_I,Q_b q).
\end{equation}

To estimate the right-hand side of (\ref{0325-01}), we use the definition of $J_2(u_I,Q_b q)$ to obtain
\begin{equation}\label{EQ:March09:201}
\begin{split}
J_2(u_I,Q_b q)  = & \sum_{T\in\T_h}\sum_{e\subset\partial T}h_T \|(\alpha\nabla u_I +\bbeta u_I)\cdot \bn_e + Q_b q \|_e^2\\
= & \sum_{T\in\T_h}\sum_{e\subset\partial T}h_T \|\alpha\nabla u_I\cdot\bn_e  - Q_{b}(\alpha\nabla u \cdot\bn_e ) + \bbeta u_I \cdot \bn_e - Q_b ( \bbeta u \cdot \bn_e)\|_e^2 \\
\le & 2 \sum_{T\in\T_h} h_T \left(\|\alpha\nabla u_I - Q_{b}(\alpha\nabla u)\|_\pT^2 + \|\bbeta u_I- Q_b (\bbeta u)\|_\pT^2 \right)\\
\le & 2 \sum_{T\in\T_h} h_T \left(\|\alpha\nabla u_I -\alpha\nabla u \|_\pT^2
+ \|\alpha\nabla u - Q_{b}(\alpha\nabla u)\|_\pT^2\right) \\
& + 2 \sum_{T\in\T_h} h_T \left( \|\bbeta u_I- \bbeta u\|_\pT^2 + \|\bbeta u - Q_b (\bbeta u)\|_\pT^2 \right).
\end{split}
\end{equation}
On each element $T\in \T_h$, let $\overline{g}$ be the average of the function $g=g(x)$ on $T$. It follows that
\begin{equation}\label{EQ:08-27-2017:001}
\|\alpha\nabla u - Q_{b}(\alpha\nabla u)\|_\pT \leq \|(\alpha - \overline{\alpha})\nabla u\|_\pT
\end{equation}
and
\begin{equation}\label{EQ:08-27-2017:002}
\begin{split}
\|\bbeta u - Q_b (\bbeta u)\|_\pT \leq & \|\bbeta u - \overline{\bbeta} \overline{u}\|_\pT\\
\leq & \|\bbeta u - \overline{\bbeta} u\|_\pT + \|\overline{\bbeta} u - \overline{\bbeta} \overline{u}\|_\pT\\
= & \|(\bbeta - \overline{\bbeta}) u\|_\pT + \|\overline{\bbeta} (u - \overline{u})\|_\pT.
\end{split}
\end{equation}
Now substituting (\ref{EQ:08-27-2017:001}) and (\ref{EQ:08-27-2017:002}) into (\ref{EQ:March09:201}) yields
\begin{equation}\label{EQ:March09:801}
\begin{split}
J_2(u_I,Q_b q) \le & 2 \sum_{T\in\T_h} h_T \left(\|\alpha\nabla (u_I-u)\|_\pT^2
+ \|(\alpha-\overline{\alpha})\nabla u\|_\pT^2\right) \\
& + 2 \sum_{T\in\T_h} h_T \left( \|\bbeta (u_I-u)\|_\pT^2 + \|(\bbeta - \overline{\bbeta}) u\|_\pT^2
+ \|\overline{\bbeta}(u - \overline{u})\|_\pT^2 \right)\\
\le & C (\|\alpha\|_\infty^2 h^2 \|u\|_2^2 + \|\bbeta\|_\infty^2 h^2 \|u\|_1^2 + \osc(\alpha, \T_h)^2 \|u\|_1^2 + \osc(\bbeta,\T_h)^2\|u\|_0^2).
\end{split}
\end{equation}
By combining (\ref{0325-01}) with (\ref{EQ:March09:801}) we obtain
$$
J_2(e_h, \eta_h)^{\frac{1}{2}} \leq C (\|\alpha\|_\infty h \|u\|_2 + \|\bbeta\|_\infty h \|u\|_1 + \osc(\alpha, \T_h) \|u\|_1 + \osc(\bbeta,\T_h) \|u\|_0).
$$
Finally, substituting the above estimate into (\ref{EQ:March01:102}) yields
$$
\|u_h-u_I\|_1 \leq C (\|\alpha\|_\infty h \|u\|_2 + \|\bbeta\|_\infty h \|u\|_1 + \osc(\alpha, \T_h) \|u\|_1 + \osc(\bbeta,\T_h) \|u\|_0),
$$
which completes the proof of the theorem.
\end{proof}

\medskip

We end this section by showing that $q_h$ indeed provides a very accurate approximation of the exact flux. For any
$m\in L^2(\E_h)$, define its $L^2$ norm by
\[
\3bar m\3bar_0=\left(\sum_{T\in\T_h} \sum_{e\subset\partial T} h_e\int_e m^2 ds\right)^\frac12.
\]

\begin{theorem}\label{thm:flux}
Under the assumptions of Theorem \ref{thm:h1}, Let $q=-(\alpha\nabla u + \bbeta u)\cdot \bn_e$ be the exact flux on $e\in\E_h$ along the direction $\bn_e$. If the exact solution $u\in H^2(\Omega)$, then the following error estimate holds true
\begin{equation}\label{fluxestimate.000}
\3bar q-q_h\3bar_0\le C (\|\alpha\|_\infty h \|u\|_2 + \|\bbeta\|_\infty h \|u\|_1 + \osc(\alpha, \T_h) \|u\|_1 + \osc(\bbeta,\T_h) \|u\|_0).
\end{equation}
In particular, if the coefficients $\alpha$ and $\bbeta$ are differentiable on each element $T\in \T_h$, then
\begin{equation}\label{fluxestimate}
\3bar q-q_h \3bar_0\le Ch\|u\|_2,
\end{equation}
where $C=C(\alpha,\bbeta)$ is a constant depending on the element-wise $W^{1,\infty}$ norm of the coefficients $\alpha$ and $\bbeta$.
\end{theorem}

\begin{proof} From $e_h=u_h-u_I$ and $\eta_h = q_h - Q_b q$ we have
\[
J_2(u_h,q_h)= s_h((e_h,\eta_h),(u_h,q_h))+s_h((u_I,Q_b q),(u_h,q_h)).
\]
Using the Cauchy-Schwarz inequality and the estimate \eqref{0325-01} one arrives at
\[
J_2(u_h,q_h)^\frac12\le J_2(e_h,\eta_h) ^\frac12+J_2(u_I,Q_b q)^\frac12\le 2 J_2(u_I,Q_b q)^\frac12,
\]
which, by the estimate (\ref{EQ:March09:801}), leads to
\begin{equation}\label{EQ:Sept:001}
J_2(u_h,q_h)^\frac12 \leq C (\|\alpha\|_\infty h \|u\|_2 + \|\bbeta\|_\infty h \|u\|_1 + \osc(\alpha, \T_h) \|u\|_1 + \osc(\bbeta,\T_h)\|u\|_0).
\end{equation}

As $q=-(\alpha\nabla u + \bbeta u)\cdot \bn_e$, we then have
\begin{equation}\label{EQ:Sept:002}
\begin{split}
J_2(u_h,q) = & \sum_{T\in\T_h} \sum_{e\subset\partial T} h_e \int_e |\alpha\nabla(u_h-u)\cdot\bn_e+ \bbeta(u_h-u)\cdot\bn_e|^2 ds\\
\le & 2 \sum_{T\in\T_h} \sum_{e\subset\partial T} h_e  \int_e \left(\|\alpha\|_{\infty, T}^2  |\nabla(u_h-u)|^2 ds + \|\bbeta\|_{\infty, T}^2 |u_h-u|^2\right) ds\\
\le & C(\|\alpha\|_\infty, \|\bbeta\|_{\infty}) (h^2 \|u\|_2^2 + \|u-u_h\|_1^2),
\end{split}
\end{equation}
where the usual trace inequality has been employed in the last line. Using the error estimate (\ref{h1-02}) we arrive at
$$
J_2(u_h,q)^{\frac12}\leq  C(\|\alpha\|_\infty, \|\bbeta\|_{\infty}) h \|u\|_2.
$$
Consequently,
\[
\3bar q-q_h\3bar_0 \le C(J_2(u_h,q_h)^\frac12+J_2(u_h,q)^\frac12)\le Ch\|u\|_2,
\]
provided that the coefficients $\alpha$ and $\bbeta$ are differentiable locally on each element $T$.
This proves the theorem.
\end{proof}

\section{Numerical Experiments}\label{numerical-experiments}
The purpose of this section is two-fold. First, we numerically verify the theoretical error estimates developed in the previous section by applying the algorithm \eqref{min-LagrangeForm2-01}-\eqref{min-LagrangeForm2-02} to the elliptic problem \eqref{ellipticbdy}-\eqref{ellipticbc}. Secondly, we shall demonstrate the significance of having a conservative numerical flux by applying the CFO algorithm to a simplified two-phase porous media flow model.

\subsection{Elliptic Equations}
The elliptic test problems are defined in two-dimensional square domains by seeking
$u \in H^1(\Omega)$ such that
\begin{equation}\label{EQ:EllipticTest}
\begin{split}
-\nabla\cdot(\alpha\nabla u)&=f,\qquad {\rm in}\  \Omega\\
u&=g, \qquad {\rm on}\ \partial\Omega.
\end{split}
\end{equation}
The CFO algorithm \eqref{min} with $r=2$ (i.e., the numerical scheme \eqref{min-LagrangeForm2-01}-\eqref{min-LagrangeForm2-02}) is implemented for each test case with numerical solutions denoted as $u_h$, $q_h$, and $\lambda_h$. The following metrics are used to measure the magnitude of the error:
\begin{eqnarray*}
\text{$L^2$-norm: }
&&\|u_h-u \|_{0}=\left(\sum_{{T}\in\T_h }\int_{{T}}|u_h-u|^2 d{T}\right)^{1/2},\\
\text{$H^1$-norm: }
&&\|u_h-u\|_1=\|\nabla(u_h-u) \|_{0},\\
\text{Residual-error: }
&& J_2(u_h,q_h)^{\frac12}=\left(\sum_{{T}\in\T_h }h_T\int_{\partial {T}}|\alpha \nabla u_h\cdot\bn_{e}+q_h|^2 d s\right)^{1/2}.
\end{eqnarray*}
\subsubsection{Test Case 1: Smooth Coefficients} In this experiment, the elliptic problem has exact solution $u=\cos(\pi x)\cos(\pi y)$ with domain $\Omega=(0,1)^2$. The coefficient $\alpha$ is the identity matrix. The right-hand side function $f$ and the Dirichlet boundary data $g$ are chosen to match the exact solution. To numerically analyze the affect of the finite element partition $\T_h$, both uniform and nonuniform meshes are considered; see Figure \ref{fig:mesh}. Table \ref{table1} illustrates the performance of the CFO algorithm for the test problem on uniform mesh. It shows clearly that the errors $\|u_h-u \|_{0}$ and $\|\lambda_h\|_{0}$ converge to zero at the rate of $h^2$, and that $\|\nabla(u_h-u) \|_{L^2}$ and the residual error $J_2(u_h,q_h)^{\frac12}$ converge at the order of $h$. The results are in good consistency with the theory. The numerical experiments on the irregular meshes as shown in Table \ref{table2} suggest the same convergence as on the uniform meshes.

\begin{figure}[H]
\centering
\includegraphics [width=0.4\textwidth]{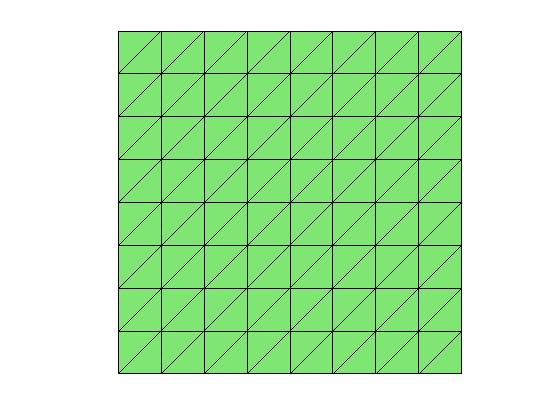}
\includegraphics [width=0.4\textwidth]{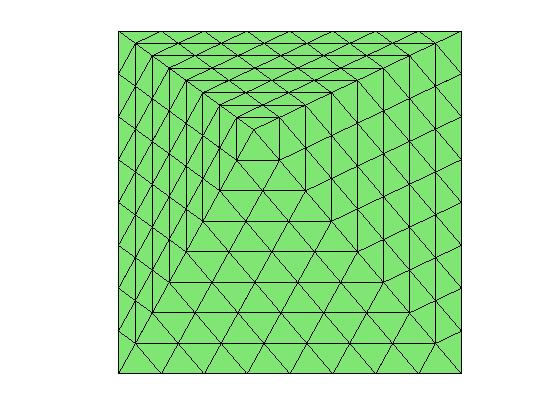}
\caption{Uniform and irregular triangle meshes with $h=1/8$.}
\label{fig:mesh}
\end{figure}
\begin{table}[h]
\begin{center}
\caption{Error and convergence performance of the CFO scheme for Test Case 1 on uniform meshes.}\label{table1}
\begin{tabular}{||c|cc|cc|cc|cc||}
\hline
h & $\|u_h-u \|_{0}$ & Order & $\|u_h-u\|_{1}$ & Order &$J_2(u_h,q_h)^{\frac12}$& Order &$\|\lambda_h\|_{0}$&Order\\
\hline
1/2 & 0.234         &          & 1.54           &           &4.49         &         &0.246&  \\
1/4 & 9.53e-2     & 1.3    & 0.860           & 0.84   &2.59         & 0.79 &0.102& 1.3\\
1/8 & 2.92e-2     & 1.7    & 0.437         & 0.98   &1.34         & 0.95 &3.06e-2& 1.7\\
1/16 &7.80e-3    & 1.9    & 0.218         & 0.99   &0.676       & 1.0   &8.12e-3& 1.9\\
1/32 &1.99e-3    & 2.0    & 0.109         & 1.0     &0.339       & 1.0   &2.07e-3& 2.0 \\
1/64 &4.99e-4    & 2.0    & 5.45e-2 & 1.0    &0.169       & 1.0    &5.18e-4& 2.0 \\
1/128 &1.25e-4  & 2.0    & 2.73e-2 & 1.0    &8.47e-2    & 1.0   &1.30e-4& 2.0 \\
\hline
\end{tabular}
\end{center}
\end{table}

\begin{table}[h]
\begin{center}
\caption{Error and convergence performance of the CFO scheme for Test Case 1 on irregular meshes.}\label{table2}
\begin{tabular}{||c|cc|cc|cc|cc||}
\hline
h & $\|u_h-u \|_{0}$ & Order & $\|u_h-u\|_{1}$ & Order &$J_2(u_h,q_h)^{\frac12}$& Order &$\|\lambda_h\|_{0}$&Order\\
\hline
1/2     & 0.147       &          & 1.09              &            &2.63      &          &0.125&  \\
1/4     & 4.07e-2   & 1.9    & 0.571            & 0.94    &1.59      & 0.73  &3.51e-2 &1.8 \\
1/8     & 1.10e-2   & 1.9    & 0.288            & 0.99    &0.836    & 0.93  &9.23e-3 &1.9 \\
1/16   &2.87e-3    & 1.9    & 0.144            & 1.0      &0.424    & 0.98  &2.36e-3 &2.0\\
1/32   &7.32e-4    & 2.0    & 7.21e-2    & 1.0      &0.213    & 0.99  &5.95e-4&2.0\\
1/64   &1.84e-4    & 2.0    & 3.60e-2   & 1.0      &0.106    & 0.99  &1.49e-4&2.0\\
1/128 &4.62e-5    & 2.0    & 1.80e-2   & 1.0      &5.32e-2 & 1.0    &3.73e-5&2.0\\
\hline
\end{tabular}
\end{center}
\end{table}

\subsubsection{Test Case 2: H\"older Continuous Coefficients} The coefficient matrix $\alpha$ in this test is given by
\begin{eqnarray}
\alpha=\begin{pmatrix}
        1+|x| & 0.5|x|^{\frac13}|y|^{\frac13} \\
        0.5|x|^{\frac13}|y|^{\frac13} & 1+|y|
        \end{pmatrix}
\end{eqnarray}
on the square domain $\Omega=(-1,1)^2$. The coefficient matrix is clearly non-smooth, but H\"older continuous. The right-hand side function and the Dirichlet boundary data are chosen to match the exact solution of $u(x,y)=\cos(\pi x)\cos(\pi y)$. Note that this example has been considered in \cite{Smears_SIAMJNA_2013,WangWang_2016}.

We use the CFO algorithm of $r=2$ (i.e., (\ref{min-LagrangeForm2-01}-\ref{min-LagrangeForm2-02}) to approximate the above elliptic problem. The meshes are obtained by first uniformly partitioning the square domain $\Omega$ into $N^2,\ N=\frac{1}{h}$, small squares and then decomposing each small square into $2$ similar triangles. The corresponding error and convergence information are reported in Table \ref{table4}. Note that $u$ denotes the exact solution of \eqref{EQ:EllipticTest}, and $(u_h;q_h)$ is the solution of the constrained minimization problem \eqref{min} with $r=2$.

We observe from Table \ref{table4} that the convergence of the algorithm in both the $H^1$ and $L^2$ norms has optimal order of $k=1$ and $k=2$, respectively. Moreover, for the flux on element edges, the convergence is also of optimal order of $k=1$. These numerical results support strongly the theoretical findings in the previous section. In fact, the numerical results outperform the theory in $H^1$, as the coefficient $\alpha$ is non-smooth nor Lipschitz continuous on each element so that no convergence of order $k=1$ can be deduced from the theory. We point out that there was no theory of optimal order of convergence in $L^2$ for CFO, though the numerics strongly suggest such a result in the $L^2$ norm. Interested readers are encouraged to study the $L^2$ convergence for the CFO algorithm.

It should be emphasized that the main advantage of the CFO algorithm (\ref{min-LagrangeForm2-01}-\ref{min-LagrangeForm2-02}) is that one obtains not only a discrete solution $u_h$ with optimal order of convergence, but also an element-wise conserving flux with optimal order of convergence as well.

\begin{table}[h]
\caption{Errors and convergence performance of the CFO scheme (\ref{min-LagrangeForm2-01}-\ref{min-LagrangeForm2-02}) for Test Case 2 with H\"older continuous coefficients.}\label{table4}
\begin{center}
\begin{tabular}{||c|cc|cc|cc||}
\hline
h & $\|u_h-u \|_0$ & Order & $\|u_h-u\|_1$ & Order &$\3bar q-q_h\3bar_0$ & Order \\
\hline
1/2     &0.769      &         & 3.27       &          &9.15     &             \\
1/4     &0.265      & 1.5    &1.74       & 0.92  &5.19     & 0.82    \\
1/8     &7.15e-2  & 1.9    &0.871      & 0.99  &2.67    & 0.96    \\
1/16   &1.79 e-2 & 2.0    & 0.436     & 1.0    &1.34    & 0.99   \\
1/32   &4.36e-3  & 2.0    & 0.218     & 1.0    &0.67    & 1.0    \\
1/64   &1.05e-3  & 2.0    & 0.109     & 1.0    &0.337  & 1.0     \\
1/128 &2.58e-4  & 2.0    & 5.47e-2  & 1.0   &0.168  & 1.0     \\
\hline
\end{tabular}
\end{center}
\end{table}

\subsubsection{Test Case 3: Discontinuous Coefficients}
In this numerical test, the domain of the elliptic problem (\ref{EQ:EllipticTest}) is chosen as the unit square $\Omega=(0,1)^2$, and the coefficient $\alpha$ is given by
\begin{eqnarray}
\alpha=\left\{
       \begin{array}{ll}
       \begin{pmatrix}
        1 & 0 \\
        0 & 1
       \end{pmatrix}, \qquad \text{ if } x<0.5,\\[10pt]
       \begin{pmatrix}
        10 & 3 \\
        3  & 1
       \end{pmatrix}, \qquad \text{ if } x\geq 0.5,
       \end{array}\right.
\end{eqnarray}
which is clearly discontinuous along the vertical line of $x=\frac12$. With properly chosen data on the right-hand side function and the Dirichlet boundary value, the exact solution of (\ref{EQ:EllipticTest}) is given by
\begin{eqnarray}
u(x,y)=\left\{
       \begin{array}{ll}
          1-2y^2+4xy+6x+2y,\qquad  \text{ if } x<0.5,\\
          -2y^2+1.6xy-0.6x+3.2y+4.3, \qquad \text{ if } x\geq 0.5.
       \end{array}\right.
\end{eqnarray}

Table \ref{table5} illustrates the performance of the CFO scheme \eqref{min-LagrangeForm2-01}-\eqref{min-LagrangeForm2-02} when applied to the present test case. The results suggest an optimal order of convergence for the numerical approximation $u_h$ in the usual $H^1$ norm, which is in great consistency with the error estimate developed in the previous section. Likewise, the numerical approximation for the flux variable $q_h$ also has an optimal order of convergence, as predicted by the convergence theory. On the other hand, the convergence in $L^2$ for $u_h$ seems to be around $k=1.8$. Again, we do not have a theory on the optimal order of convergence in $L^2$.
\begin{table}[h]
\caption{Errors and convergence performance of the CFO algorithm (\ref{min-LagrangeForm2-01}-\ref{min-LagrangeForm2-02}) for Test Case 3 with a discontinuous coefficient.}\label{table5}
\begin{center}
\begin{tabular}{||c|cc|cc|cc||}
\hline
h & $\|u_h-u \|_0/\|u \|_0$ & Order & $\|u_h-u\|_1/\|u\|_1$ & Order &$\3bar q-q_h\3bar_0/\3bar q\3bar_0$ & Order \\
\hline
1/4      &2.43e-03   &0.00     &6.57e-02   &0.00     &7.59e-02    &0.00\\
1/8      &6.71e-04   &1.86     &3.28e-02   &1.00     &3.04e-02    &1.32\\
1/16     &2.08e-04   &1.69     &1.64e-02   &1.00     &1.32e-02    &1.20\\
1/32     &6.16e-05   &1.75     &8.17e-03   &1.00     &6.12e-03    &1.11\\
1/64     &1.79e-05   &1.78     &4.08e-03   &1.00     &2.96e-03    &1.05\\
1/128    &5.19e-06   &1.79     &2.04e-03   &1.00     &1.46e-03    &1.02\\
1/256    &1.47e-06   &1.82     &1.02e-03   &1.00     &7.30e-04    &1.01\\
\hline
\end{tabular}
\end{center}
\end{table}

\subsubsection{Test Case 4: Discontinuous Coefficients}
In this test case, the domain $\Omega=(-1,1)^2$ is split into four subdomains $\Omega = \bigcup_{i=1}^{4} \Omega_i$ by the $x$ and $y$ axis. The diffusion coefficient $\alpha$ is given by
\begin{eqnarray*}
\alpha= \begin{pmatrix}
        \alpha_i^x & 0 \\
        0 & \alpha_i^y
       \end{pmatrix}, \text{ if } (x,y) \in \Omega_i,\\
\end{eqnarray*}
and the exact solution is given by $u(x,y)=\alpha_i \sin(2\pi x)\sin(2\pi y)$. Here the values of the coefficient $\alpha_i^x,\,\alpha_i^y$ and $\alpha_i$ are specified in Table \ref{table6}. It is clear that the diffusion coefficient $\alpha$ is discontinuous across the lines $x = 0$ and $y = 0$.

\begin{table}[h]
\caption{Test Case 4: Parameter values for the diffusion coefficients and the exact solution.}\label{table6}
\begin{center}
\begin{tabular}{|l|l|}
\hline
&\\[-7pt]
$\alpha^x_4 = 0.1$ & $\alpha^x_3 = 1000$   \\[1pt]
$\alpha^y_4 =0.01$ & $\alpha^y_3 = 100 $   \\[1pt]
$\alpha_4   = 100$ & $\alpha_3   = 0.01$   \\[3pt]
\hline
&\\ [-7pt]
$\alpha^x_1 = 100$   & $\alpha^x_2 = 1  $\\[1pt]
$\alpha^y_1 = 10 $   & $\alpha^y_2 = 0.1$\\[1pt]
$\alpha_1   = 0.1$   & $\alpha_2   = 10 $\\ [3pt]
\hline
\end{tabular}
\end{center}
\end{table}

Table \ref{table7} presents the numerical performance of the CFO scheme \eqref{min-LagrangeForm2-01}-\eqref{min-LagrangeForm2-02} when applied to the present test case. The results suggest an optimal order of convergence for the numerical approximation $u_h$ in the usual $H^1$ norm and the flux variable $q_h$ in $L^2$. The numerical results are in great consistency with the error estimate developed in the previous section. It should be noted that the numerical results strongly suggest an optimal order of convergence in $L^2$ for the primal variable $u_h$.

\begin{table}[h]
\caption{Relative error and convergence performance of the CFO algorithm (\ref{min-LagrangeForm2-01}-\ref{min-LagrangeForm2-02}) for Test Case 4 with discontinuous coefficients.}\label{table7}
\begin{center}
\begin{tabular}{||c|cc|cc|cc||}
\hline
h & $\|u_h-u \|_0/\|u \|_0$ & Order & $\|u_h-u\|_1/\|u\|_1$ & Order &$\3bar q-q_h\3bar_0/\3bar q\3bar_0$ & Order \\
\hline
1/4       &8.46e-01   &0.00     &8.11e-01   &0.00     &6.15e-01    &0.00\\
1/8       &4.67e-01   &0.86     &5.13e-01   &0.66     &3.84e-01    &0.68\\
1/16      &1.75e-01   &1.42     &2.45e-01   &1.07     &1.78e-01    &1.11\\
1/32      &5.07e-02   &1.79     &1.10e-01   &1.16     &7.69e-02    &1.21\\
1/64      &1.34e-02   &1.92     &5.15e-02   &1.09     &3.48e-02    &1.14\\
1/128     &3.42e-03   &1.97     &2.49e-02   &1.05     &1.66e-02    &1.07\\
1/256     &8.62e-04   &1.99     &1.23e-02   &1.02     &8.15e-03    &1.03\\
\hline
\end{tabular}
\end{center}
\end{table}

\subsubsection{The Lagrange multiplier $\lambda_h$}
The CFO algorithm involves two essential ideas in flux approximation: (1) the satisfaction of the mass conservation equation, and (2) the minimization of the object function $J_r(v,p)$ defined as in (\ref{EQ:functional}). As the value of the PDE coefficients may vary from element to element, one may modify the object function as follows by placing a weight $\tau_T$ on each element:
\begin{equation}\label{EQ:functional-new}
J_r^*(p,v):= \frac{1}{r}\sum_{T\in\T_h} \sum_{e\subset\pT} \tau_T h_T \int_e |p+\alpha^*\nabla v \cdot \bn_e +\bbeta^* v\cdot\bn_e|^r ds.
\end{equation}
For the CFO algorithm \ref{min}, it can be seen from \eqref{min-LagrangeForm-01}-\eqref{min-LagrangeForm-02} that the weight parameter $\tau_T$ is automatically adjusted by the Lagrange multiplier $\lambda_h$ on each element, as the weight $\tau_T$ can be easily absorbed by $\lambda_h$ through the same scaling on each element. Therefore, the CFO algorithm is quite robust in the minimization part.

Figure \ref{fig:lambda} shows the surface plot of the Lagrange multiplier $\lambda_h$ for each test case. It can be seen that $\lambda_h$ is quite sensitive to the continuity and smoothness of the true solution. We conjecture that $\lambda_h$ would play the role of a posteriori error estimator in adaptive grid local refinements.

\begin{figure}
\centering
\subfigure[Smooth coefficient (left) and H\"older continuous coefficient (right)]{
\label{Fig.sub.1.lam}
\includegraphics [width=0.35\textwidth]{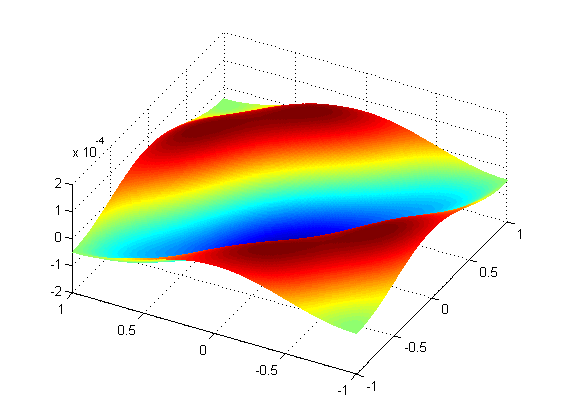}
\includegraphics [width=0.35\textwidth]{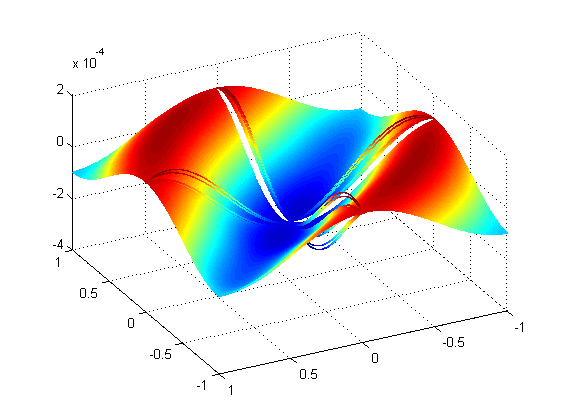}
}\\
\subfigure[Discontinuous coefficients: test case 3 (left) and test case 4 (right)]{
\label{Fig.sub.2.lam}
\includegraphics [width=0.35\textwidth]{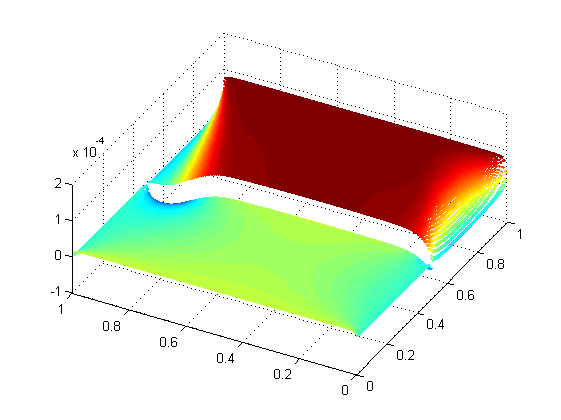}
\includegraphics [width=0.35\textwidth]{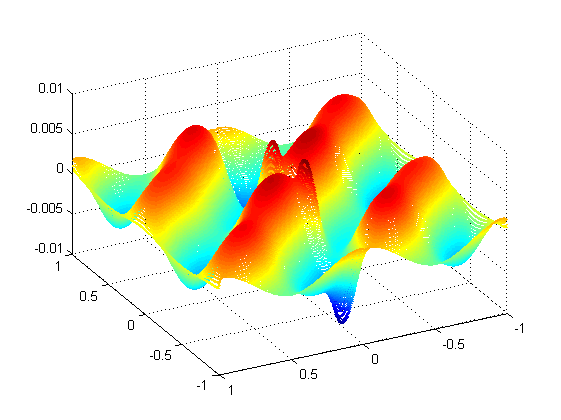}
}\\
\caption{The solution profile for the Lagrange multiplier $\lambda_h$ on a partition of size $128\times 128$
 arising from the CFO scheme (\ref{min-LagrangeForm2-01}-\ref{min-LagrangeForm2-02}).}
\label{fig:lambda}
\end{figure}

%\begin{figure}[h]
%  \begin{center}
%  \includegraphics[width=0.35\textwidth]{Lambda-Smooth-128x128.png}
%  \includegraphics[width=0.35\textwidth]{Lambda-Holder-128x128.png}
%  \caption{The solution profile for the Lagrange multiplier $\lambda_h$ on a partition of size $128\times 128$: smooth coefficient (left), H\"older continuous coefficient (right).}
%  \label{Label:Lambda:001}
% \end{center}
%  \end{figure}

%\begin{figure}[h]
%\centering
%\includegraphics [width=0.35\textwidth]{Lambda-Discon01-128x128.png}
%\hskip-50pt
%\includegraphics [width=0.35\textwidth]{Lambda-Discon02-128x128.png}
%\caption{The solution profile for the Lagrange multiplier $\lambda_h$ on a partition of size $128\times 128$ with discontinuous coefficients.}
%\label{Label:Lambda:002}
%\end{figure}

\subsection{A Two-Phase Flow in Porous Media}
We consider a simplified two-phase flow problem in porous media which seeks a saturation function $S$ and fluid pressure $p$ satisfying
\begin{eqnarray}
-\nabla\cdot(\lambda(S)\kappa(x,y) \nabla p)&=&0, \qquad {\rm in}\  \Omega=(0,1)^2, \label{two-phase-elliptic}\\
\frac{\partial S}{\partial t} +\d(\bv f(S))&=&0,\qquad  t\in{\bf R}^+, \label{two-phase-transport}
\end{eqnarray}
with the boundary condition
\begin{eqnarray}
  p(0,y)=1, & & \; p(1,y)=0,\\
  \kappa(x,y)\nabla p(x,0)\cdot \bn & = &\kappa(x,y)\nabla p(x,1)\cdot \bn=0,
\end{eqnarray}
for the fluid pressure $p$ and the following initial and boundary conditions for
the saturation:
\begin{eqnarray}
S(0,y,t)&=&1,\qquad t \geq 0, \ y\in (0,1),\\
S(x,y,0) &=& 0,\qquad (x,y)\in (0,1)^2.
\end{eqnarray}
Here $\bv=-\lambda(S)\kappa(x,y) \nabla p$ is the Darcy's velocity, $\kappa(x,y)$ is the permeability of the porous media, $f(S)$ is the fractional flow function, and $\lambda(S)$ is the total mobility. This model problem has been studied in several existing literatures including \cite{Bush_SIAMJSC_2013,Bush_JCAM_2014}. In our numerical study, we choose a high-contrast, heterogeneous permeability function
\begin{eqnarray*}
\kappa(x,y)=\frac{1}{0.25-0.999(x-x^2)\sin(11.2\pi x)}\cdot\frac{1}{0.25-0.999(y-y^2)\sin(5.2\pi y)},
\end{eqnarray*}
the flow function
\begin{eqnarray*}
 f(S)=\frac{S^2}{S^2+(1-S)^2/5},
\end{eqnarray*}
and the total mobility function as
\begin{eqnarray*}
 \lambda(S)=S^2+(1-S)^2/5.
\end{eqnarray*}

The permeability profile is plotted in a logarithmic scale in Figure \ref{fig:permeability}. The figure shows that $\kappa(x,y)$ has a channelized pattern and is highly heterogenous. Note that when $\lambda(S)=1$, the equations reduce to a single-phase flow model.

\begin{figure}
\centering
\includegraphics [width=0.7\textwidth]{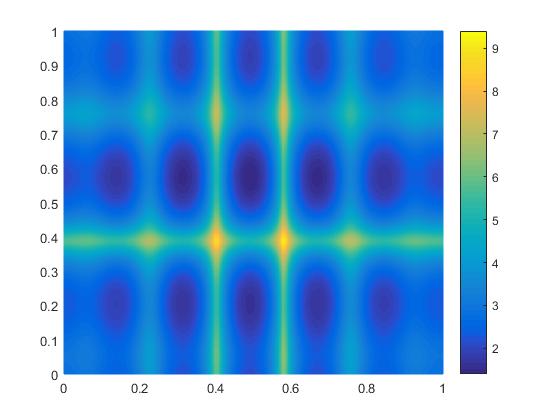}
\caption{Permeability profile plotted in the logarithmic scale.}
\label{fig:permeability}
\end{figure}

We use an operator splitting technique \cite{Aziz_Chapman_1979} to solve the above system. That is, we substitute the saturation at the previous time step
into \eqref{two-phase-elliptic} to compute the pressure $p$ and the Darcy's velocity $\bv$. Then we solve the transport equation \eqref{two-phase-transport} by an explicit time stepping scheme.

We first explain the numerical solution for \eqref{two-phase-elliptic} which is elliptic for any given saturation $S$. In existing literatures, there are many numerical methods available for approximating this elliptic equation, including the classical Galerkin finite element method (FEM). However, the straightforward numerical flux $\bv_h=\lambda(S)\kappa(x,y)\nabla p_h$ obtained from the classical FEM is known to be discontinuous across the element interface. Consequently, it is challenging to design conservative numerical schemes for the transport equation (\ref{two-phase-transport}) based on such numerical fluxes. To overcome this difficulty, we solve the flow equation (\ref{two-phase-elliptic}) by using the CFO algorithm proposed in Section \ref{sectionccfv}.

To discretize the transport equation, we first integrate (\ref{two-phase-transport}) with respect to time and then over each control element $T\in \T_h$. We apply the left-end point quadrature rule to its second term in time and use integration by parts to arrive at the following scheme \cite{Bush_SIAMJSC_2013,Bush_JCAM_2014}
\begin{eqnarray}\label{EQ:Saturation-Discrete}
 |{T}|(S_n^{{T}}-S_{n-1}^{{T}})+\Delta t \int_{\partial {T}}\bv_h\cdot\bn_e f(S_{n-1}^{{T}})ds=0.
\end{eqnarray}
From the numerical scheme (\ref{EQ:Saturation-Discrete}), the saturation $S$ is defined on each element and the numerical flux $\bv_h\cdot\bn_e$ is needed on each edge. The continuity of this numerical flux becomes to be necessary for the mass conservation of the scheme. Our CFO finite element method \eqref{min-LagrangeForm2-01}-\eqref{min-LagrangeForm2-02} was designed to provide not only a continuous numerical flux but also one that conserves mass locally on each element.

An upwinding strategy was employed for evaluating the boundary integral in (\ref{EQ:Saturation-Discrete}) as follows
\begin{eqnarray*}
\int_{\partial {T}}\bv_h\cdot\bn_e f(S_{n-1}^{{T}})ds = \sum_{e\in \partial{T}}\int_{e}\bv_h\cdot\bn_e f^*(S_{n-1}^{{T}})ds,
\end{eqnarray*}
with
\begin{equation*}
\int_{e}\bv_h\cdot\bn_e f^*(S_{n-1}^{{T}})ds := \left\{
\begin{array}{lll}
|e|\bv_h\cdot\bn_e f(S_{n-1}^{{T}_{e,L}}),\quad  \text{ if } \bv_h\cdot\bn_e\geq0\\
~\\
|e|\bv_h\cdot\bn_e f(S_{n-1}^{{T}_{e,R}}), \quad \text{ otherwise},
\end{array}
\right.
\end{equation*}
where $T_{e,L}$ and $T_{e,R}$ stand for the `left' (or the upwind side) and `right' side (or the downwind) element of edge $e$ respectively, according to the sign of $\bv_h\cdot\bn_e$ as shown in Figure \ref{fig:upwind}.
Note that the numerical flux $\bv_h\cdot\bn_e$ is constant on each edge and can be obtained directly from solving the elliptic part of the system.

\begin{figure}
\centering
\includegraphics [width=0.5\textwidth]{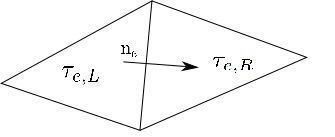}
\caption{${T}_{e,L}$ and ${T}_{e,R}$ stand for the `left' and `right' side element of edge $e$ according to the sign of $\bv_h\cdot\bn_e$.}
\label{fig:upwind}
\end{figure}

The saturation profiles for the two-phase flow at time T=0.002, 0.01 and 0.02 are shown in Figure \ref{fig:twophase-FOG}. In these plots, two finite element partitions of size $64\times64$ (Figure \ref{Fig.sub.1}) and $128\times128$ (Figure \ref{Fig.sub.2}) are employed. The time steps of $\Delta t = 1.0e-5$ (Figure \ref{Fig.sub.2}) and $\Delta t = 1.0e-6$ (Figure \ref{Fig.sub.3}) are used in this simulation. It can be seen that the numerical results corresponding to the two different partitions and time steps are very close to each other, and they are considered to be convergent. These computational results match the physics greatly, and hence confirm the effectiveness and robustness of the locally conservative fluxes provided by the CFO finite element method \eqref{min-LagrangeForm2-01}-\eqref{min-LagrangeForm2-02}.

\begin{figure}
\centering
\subfigure[Saturation profiles at time $T=0.002$, $0.01$ and $0.02$ using the partition of size $64\times64$, with $\Delta t =1.0\times10^{-5}$.]{
\label{Fig.sub.1}
\includegraphics [width=0.33\textwidth]{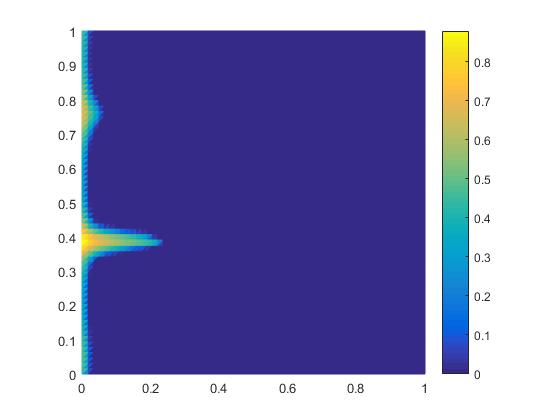}
\hskip-50pt
\includegraphics [width=0.33\textwidth]{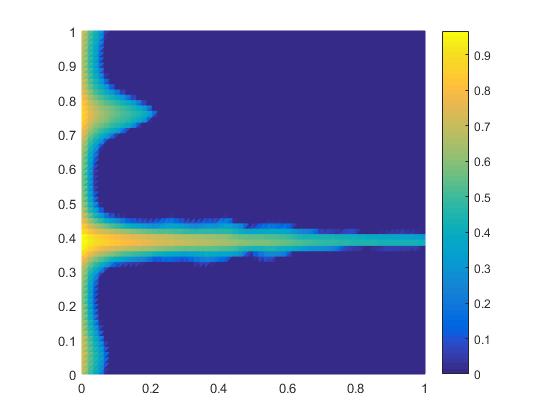}
\hskip-50pt
\includegraphics [width=0.33\textwidth]{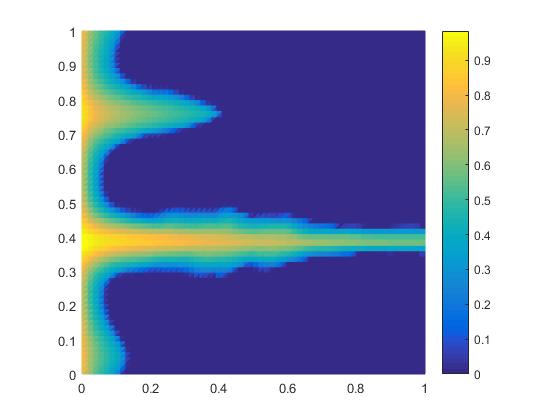}
}\\
\subfigure[Saturation profiles at $T=0.002$, $0.01$ and $0.02$ using the partition of size $128\times128$, with $\Delta t =1.0\times10^{-5}$.]{
\label{Fig.sub.2}
\includegraphics [width=0.33\textwidth]{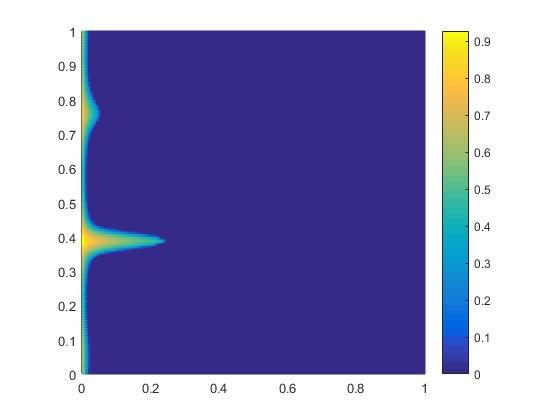}
\hskip-50pt
\includegraphics [width=0.33\textwidth]{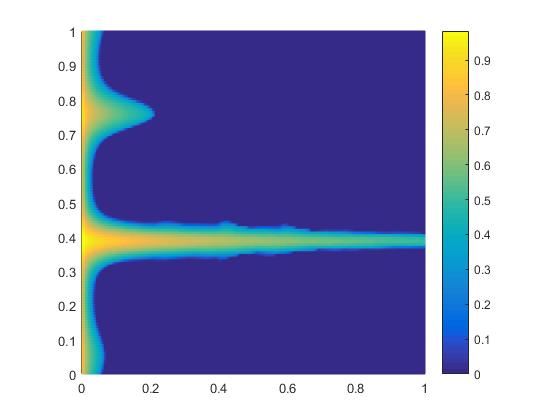}
\hskip-50pt
\includegraphics [width=0.33\textwidth]{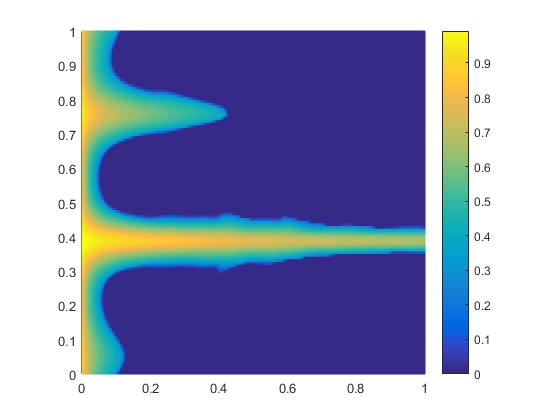}
}\\
\subfigure[Saturation profiles at $T=0.002$, $0.01$ and $0.02$ using the partition of size $128\times128$, with $\Delta t =1.0\times10^{-6}$.]{
\label{Fig.sub.3}
\includegraphics [width=0.33\textwidth]{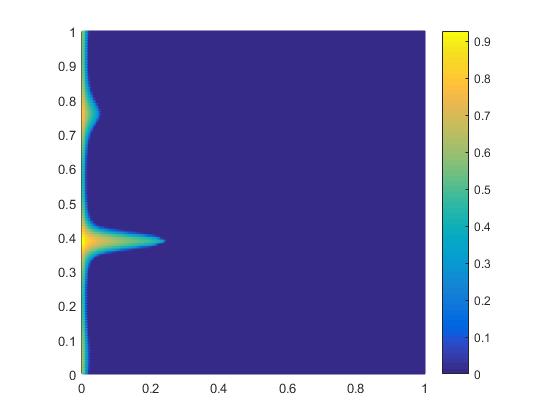}
\hskip-50pt
\includegraphics [width=0.33\textwidth]{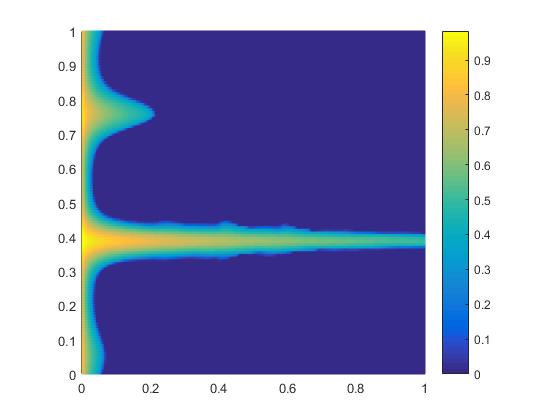}
\hskip-50pt
\includegraphics [width=0.33\textwidth]{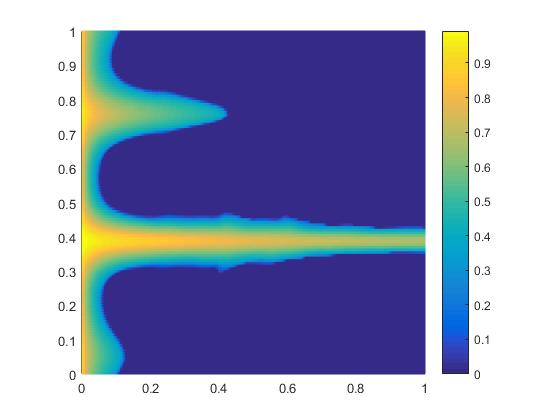}
}\\
\caption{Saturation profiles for the two-phase flow in porous media obtained using the CFO scheme \eqref{min-LagrangeForm2-01}-\eqref{min-LagrangeForm2-02}.}
\label{fig:twophase-FOG}
\end{figure}

\section*{Acknowledgements}
The authors are grateful to Todd Arbogast, Malgorzata Peszynska, Ralph Showalter, and Son-Young Yi for a helpful discussion of this work during the SIAM 2017 Mathematical and Computational Geoscience conference in Erlangen Germany. This discussion has led to the name of ``flux optimization" for the numerical scheme developed in this paper.

\end{document}